\documentclass[a4paper,11pt]{amsart}
\usepackage{amssymb,latexsym,amsmath,amsthm}
\usepackage{enumerate}
\usepackage{graphicx,hyperref}

\DeclareMathOperator{\rk}{rk}
\DeclareMathOperator{\CRdim}{CR-dim}
\DeclareMathOperator{\CRcodim}{CR-codim}

\newcommand{\bC}{\mathbb{C}}
\newcommand{\bR}{\mathbb{R}}

\DeclareMathOperator\tr{tr}

\newtheorem{Th}{Theorem}[section]
\newtheorem{Prop}[Th]{Proposition}
\newtheorem{Cor}[Th]{Corollary}
\newtheorem{Lem}[Th]{Lemma}
\theoremstyle{definition}
\newtheorem{Def}[Th]{Definition}
\newtheorem{Rem}[Th]{Remark}
\newtheorem{Not}[Th]{Notation}

\newcommand{\bt}{\begin{Th}\ \ }
\newcommand{\et}{\end{Th}}
\newcommand{\bp}{\begin{Prop}\ \ }
\newcommand{\ep}{\end{Prop}}
\newcommand{\bc}{\begin{Cor}\ \ }
\newcommand{\ec}{\end{Cor}}
\newcommand{\bl}{\begin{Lem}\ \ }
\newcommand{\el}{\end{Lem}}
\newcommand{\bd}{\begin{Def}\ \ }
\newcommand{\ed}{\end{Def}}
\newcommand{\br}{\begin{Rem}\ \ }
\newcommand{\er}{\end{Rem}}

\newcommand{\arr}{\begin{array}{rlll}}
\newcommand{\ea}{\end{array}}

\numberwithin{equation}{section}

\newcommand{\di}{\mathrm{d}}

\newcommand{\ii}{\mathrm{i}}
\newcommand{\ee}{\mathrm{e}}
\begin{document}
\title[Associated families of immersions of $3$-dim CR manifolds]{Associated families of immersions\\ of three dimensional CR manifolds \\in Euclidean spaces}
\author[A.~Altomani]{Andrea Altomani}
\address[A.~Altomani]{University of Luxembourg\\ Research Unit in Mathematics\\ rue Coudenhove-Kalergi 6\\ L-1359 Luxembourg}
\email{andrea.altomani@uni.lu}
 \author[M.-A.~Lawn]{Marie-Am\'elie Lawn}
\address[M.-A.~Lawn]{Philipps Universit\"at Marburg\\Department of Mathematics FB12\\
Hans-Meerwein-Str.\\
D-35032 Marburg \\Germany}
\email{amelielawn@mathematik.uni-marburg.de}

\begin{abstract}
We consider isometric immersions in arbitrary codimension of three-dimensional strongly pseudoconvex pseudo-hermitian CR manifolds into the Euclidean space $\mathbb{R}^n$ and generalize in a natural way the notion of associated family. We show that the existence of such deformations turns out to be very restrictive and we give a complete classification.
\end{abstract}
\keywords{Isometric immersions, CR-pluriharmonic immersions, strongly pseudoconvex CR manifolds}
\subjclass[2010]{primary: 53C42, secondary: 53A07, 32V10, 53D10}

\date{}
\maketitle
\section{Introduction}
A well-known property of isometric immersions $f: M\rightarrow \mathbb{R}^3$ of minimal surfaces in Euclidean 3-space is the existence of the so called strong associated family, i.e.\ a one-parameter family $f_{\varphi}$, $\varphi\in S^1$, with $f_0=f$,  of isometric immersions ``rotating the differential'' and preserving, at every point, the tangent plane and the Gau\ss{} map. More precisely, denoting by $r_\varphi\colon TM\to TM$ the rotation of the tangent plane of angle $\varphi$, an associated family of immersion is a smooth family $f_\varphi$ satisfying
\[
\psi_\varphi\circ df_{\varphi}(X)=  df_0(r_{\varphi}X),
\]
for all $X\in TM$ and for some parallel bundle isomorphism $\psi_{\varphi}\colon f_{\varphi}^*T\mathbb{R}^n\rightarrow f^*_0T\mathbb{R}^n$.
Moreover it is known that minimality, the existence of an associated family  and the harmonicity of $f$  are three equivalent conditions.

 This definition of strong associated family can be naturally extended without any significant modification to the case where $M$ is a complex manifold of arbitrary dimension and  $f\colon M\to\mathbb{R}^n$ is a pluriconformal immersion, i.e. it is conformal along any complex curve in $M$ (see \cite{E}). The existence of a strong associated family is then equivalent to the pluriharmonicity of the immersion, i.e. to the fact that the restriction of $f$ to any smooth complex curve is minimal in $\mathbb{R}^n$.

A generalization to a larger class of immersions is given by submanifolds with parallel pluri-mean curvature (shortly \emph{ppmc}, see for instance \cite{E,BEFT}). A pluriconformal immersion $f\colon M\rightarrow\mathbb{R}^n$  of a complex manifold $M$ is ppmc if the $(1,1)-$part $A^{1,1}$ of the complexified second fundamental form is parallel with respect to the connection in the normal bundle.

For example, in the case of a Riemann surface $M$, we have $A^{1,1} = g(\,\cdot\,,\,\cdot\,)H$, where $g$ is the metric induced by the immersion and $H = \tr A$ is the mean curvature vector field, thus ppmc immersions correspond to surfaces with constant mean curvature.

Ppmc immersions can also be characterized by the existence of a \emph{weak} associated family $\{f_\varphi\}$in the following sense. Let $f_\varphi$ be a family of pluriconformal immersions of a complex manifold $M$ into $\bR^n$, and $A^{f_{\varphi}}$  their second fundamental forms. Then $f_\varphi$ is a weak associated family if there exists a family of parallel bundle isomorphisms $\iota_{\varphi}\colon N^{\varphi}\rightarrow N^0$, where  $N^{\varphi}$ is the normal bundle of $f_{\varphi}$, such that:
\begin{align*}
\iota_{\varphi}(A^{f_\varphi}(X,Y))=
A^{f_0}(r_{\varphi}X,r_{\varphi}Y),
\end{align*}
for all $X,Y\in TM$. Examples are of course cmc surfaces, as well as the (pluri)minimal submanifolds of $\bR^n$ described above, which are exactly those for which  $A^{1,1}=0$.  In general, ppmc submanifolds are exactly the submanifolds whose Gau\ss{} maps are pluriharmonic.

Our aim is to generalize this notion of associated family to immersions $f\colon M\to\mathbb{R}^n$ of odd dimensional manifolds, and in the present paper we will consider three-dimensional strongly pseudoconvex pseudohermitian CR manifolds $(M, T^{1,0}M,\theta)$. The basic idea is to let $HM=\Re T^{1,0}M\subset TM$ play the role corresponding to the tangent bundle in the complex case.
Strongly pseudoconvex CR manifolds admit a natural family of metrics, the \emph{Webster metrics}, all of them conformally equivalent when restricted to the subbundle $HM$.
We found natural to replace the conformality condition with  the request that $f$ be an isometry with respect to some choice of the Webster metric.
For the definition of associated family, we replace rotations  with maps that act like rotations on $HM$, and fix the characteristic direction.
More in detail, if $\theta$ is a pseudohermitian structure and $T$ its Reeb vector field, for $Z\in T^{1,0}M$ we consider the ``rotations'' $r_\varphi\colon T^\bC M\to T^\bC M$, $\varphi\in\bR$:
 \begin{align*}
r_{\varphi}Z=e^{\ii\varphi}Z,\quad r_{\varphi}\bar{Z}=e^{-\ii\varphi}\bar{Z},\quad r_{\varphi}T=T,
\end{align*}
and we look for associated families of immersions which are \emph{isometric} with respect to the Webster metric on $M$.

It is straightforward to show that strong associated families cannot exist in this context.
We investigate then the existence of weak associated families. It turns out that this case is much more restrictive than the classical one for complex manifolds: in fact we show that the only three dimensional strongly pseudo-convex CR manifolds admitting such deformations are the standard sphere and the affine space, with the pseudohermitian structure of a cylinder in $\bC^2$.

In a previous paper \cite{AL} we proved that the only \emph{CR pluriharmonic} isometric immersions of a three-dimensional CR manifold $M$ into $\bR^4$ are the standard embedding of the sphere and of the cylinder. Thus, in contrast with the complex case, CR pluriharmonicity and existence of an associated family are independent notions.\\

The main part of this work was completed during a stay of the second author at
the Mathematics Research Unit of the University of Luxembourg. She acknowledges with gratitude the generous support.

\section{Three dimensional CR manifolds}

We recall the general definitions of CR manifolds, pseudohermitian structures, and the Webster metric.

\subsection{CR manifolds and pseudohermitian structures}
\begin{Def}
A \emph{CR manifold} is the datum of a smooth manifold $M$, and of a complex subbundle $T^{1,0}M\subset T^\bC M$ of the complexification $T^\bC M=\bC\otimes TM$ of the tangent bundle of $M$, such that
\begin{gather*}
 T^{1,0}M\cap\overline{T^{1,0}M}=0,\\
[\Gamma(T^{1,0}M),\Gamma(T^{1,0}M)]\subset\Gamma(T^{1,0}M).
\end{gather*}

The \emph{CR dimension} of $M$ is $\CRdim M=\rk_\bC T^{1,0}M$, and the \emph{CR codimension} of $M$ is $\CRcodim M=\dim_\bR M-2\CRdim M$.

A CR manifold of codimension one is said to be of \emph{hypersurface type}.
\end{Def}
CR manifolds of CR codimension zero are exactly complex manifolds.

%
%
%


\begin{Def}
A \emph{pseudohermitian structure} on a CR manifold $(M,T^{1,0}M)$ of hypersurface type is a nowhere vanishing $1$-form $\theta$ such that $\theta=\bar\theta$ and $\ker\theta=T^{1,0}M\oplus T^{0,1}M$.

A CR manifold with a pseudohermitian structure is a \emph{pseudohermitian CR manifold}.
\end{Def}
Notice that any two pseudohermitian structures $\theta$, $\theta'$ are related by $\theta'=\lambda\theta$ for some nonvanishing smooth function $\lambda$.

\begin{Def}
Let $(M,T^{1,0}M)$ be a CR manifold of hypersurface type, and $\theta$ a pseudohermitian structure on $M$. The \emph{Levi form} associated to $\theta$ is the hermitian symmetric form on $T^{1,0}M$:
\begin{align*}
\mathcal{L}_{\theta}\colon T^{1,0}M\times T^{1,0}M&\longrightarrow\bC\\
  (Z,W)&\longmapsto\mathcal{L}_{\theta}(Z,W)=\ii\,\di\theta(Z,\bar W).
\end{align*}
If $x$ is a point of $M$ and the Levi form at $x$ is nondegenerate (resp. definite), then $M$ is said to be \emph{Levi nondegenerate at $x$} (resp. \emph{strongly pseudoconvex at $x$}). A CR manifold is \emph{Levi nondegenerate} (resp. \emph{strongly pseudoconvex}) if it is Levi nondegenerate (resp. strongly pseudoconvex) at every point.
\end{Def}

\begin{Rem}
A pseudohermitian structure is a (nondegenerate) contact form on $M$ if and only if the Levi form is nondegenerate.
\end{Rem}

\begin{Rem}
Levi nondegeneracy and strong pseudoconvexity are independent from the choice of the pseudohermitian structure. Moreover, by replacing $\theta$ by $-\theta$ if needed, we can and will assume that a strongly pseudoconvex manifold has a positive definite Levi form.
\end{Rem}

\begin{Def}
Let $(M,T^{1,0}M,\theta)$ be a Levi nondegenerate pseudohermitian CR manifold. The unique vector field $T$ on $M$ such that
\[
\theta(T)=1,\qquad i_T(\di\theta)=0
\]
is the \emph{Reeb vector field}.
\end{Def}


\begin{Def}
Let $(M,T^{1,0}M,\theta)$ be a Levi nondegenerate pseudohermitian CR manifold.  The \emph{Webster pseudo-Riemannian metric} associated to $\theta$ is the symmetric nondegenerate bilinear form
\[ g_{\theta}\colon TM\times TM\longrightarrow\bR \]
defined by
\begin{align*}
g_\theta(X,Y)=\Re\mathcal{L}_\theta(Z,W),\quad g_\theta(X,T)=0,\quad g_\theta(T,T)=1
\end{align*}
where
\[X=\frac{Z+\bar Z}{\sqrt{2}},\quad Y=\frac{W+\bar W}{\sqrt{2}},\quad Z,W\in T^{1,0}M \]
and $T$ is the Reeb vector field.

If $M$ is strongly pseudoconvex and $\theta$ is chosen so that $\mathcal{L}_{\theta}$ is positive definite, then the Webster pseudo-Riemannian metric is positive definite and is called the \emph{Webster metric}.
\end{Def}
\begin{Not}\label{not:product}
We will denote by
\[ \langle\,\cdot\,.\,\cdot\,\rangle_\theta\colon T^\bC M\times T^\bC M\to\bC,\]
the bilinear symmetric extension of the Webster metric $g_\theta$ to  $T^\bC M$,
and we set:
\[ \|U\|^2_\theta=\langle U,\bar U\rangle_\theta, \qquad U\in T^\bC M. \]
We will omit the subscript $\theta$ when the choice of the pseudohermitian structure is clear from the context.
\end{Not}

\subsection{Three-dimensional CR manifolds}

Let $(M, T^{1,0}M, \theta)$ be a three-dimensional Levi nondegenerate pseudohermitian CR manifold. Notice that in the three-dimensional case Levi nondegeneracy is equivalent to strong pseudoconvexity. We always assume that the pseudohermitian structure is chosen in such a way that the Levi form is positive definite.

The bundle $T^{1,0}M$ has complex dimension equal to one, hence locally there exists a complex vector field $Z$ generating $T^{1,0}M$ at every point. Its complex conjugate
$\bar{Z}$ generates $\overline{T^{1,0}M}$.

The Levi form $\mathcal{L}_\theta$ is completely determined by the value
\[\mathcal{L}_{\theta}(Z,Z)=\ii\,\di\theta(Z,\bar Z).\]
By the pseudoconvexity condition, $\mathcal{L}_{\theta}(Z,Z)$ is everywhere positive.
Upon replacing $Z$ with a scalar multiple we can assume that $\mathcal{L}_{\theta}(Z,{Z})=1$.

\begin{Def}
A \emph{pseudohermitian local frame} on a three-dimensional strongly pseudoconvex pseudohermitian CR manifold $(M, T^{1,0}M, \theta)$ is a local frame $(Z, \bar Z, T)$ with $\mathcal{L}_{\theta}(Z,{Z})=1$ and $T$ equal to the Reeb vector field.
\end{Def}


\begin{Rem}
The vectors of a pseudohermitian local frame $(Z,\bar Z, T)$ satisfy:
\begin{align*}
\langle Z,Z\rangle_\theta&=\langle \bar Z,\bar Z\rangle_\theta=0, & \langle Z,T\rangle_\theta&=\langle \bar Z,T\rangle_\theta=0,& \langle Z,\bar Z\rangle_\theta=1.
\end{align*}
Moreover the real vectors
\[ X=\frac{Z+\overline{Z}}{\sqrt 2},\qquad Y=\frac{\ii(Z-\overline{Z})}{\sqrt 2} \]
together with $T$ are a real orthonormal frame for $g_\theta$.
\end{Rem}

\begin{Not}
If $(Z,\bar Z, T)$ is a pseudohermitian local frame of $M$, we denote by $(\zeta,\bar\zeta, \theta)$ the local frame of $T^*{}^\bC M$ dual to $(Z,\bar Z, T)$ and we define three complex valued smooth functions $a$, $b$, and $c$ on $M$ by:
\begin{align}
\ii[Z,\bar Z]&=T+aZ+\bar a \bar Z, & a&=\ii\zeta[Z,\bar Z], \label{eq_for_a}\\
[Z,T]&= bZ+\bar c \bar Z, & b&=\zeta[Z,T], \label{eq_for_b_barc} \\
[\bar Z,T]&=cZ+\bar b\bar Z, & c&=\zeta[\bar Z,T]. \label{eq_for_c_barb}
\end{align}
\end{Not}
Straightforward computations yield:
\begin{align*}
\di\theta&=\ii\,\zeta\wedge\bar\zeta, \\
\di\zeta&=\ii a \,\zeta\wedge\bar\zeta-b\,\zeta\wedge\theta-c\,\bar\zeta\wedge\theta, \\
\di\bar\zeta&=\ii\bar a\,\zeta\wedge\bar\zeta-\bar c\,\zeta\wedge\theta-\bar b\,\bar\zeta\wedge\theta.
\end{align*}
Moreover, from the Jacobi identity for $Z$, $\bar{Z}$ and $T$ we have
\begin{align}\label{eq_jacobi}
b+\bar b&=0,&
\ii Zc-\ii\bar Zb+Ta-ab-\bar ac&=0.
\end{align}

\begin{Rem}\label{rem:change_of_frame}
If $(Z,\bar Z, T)$ is a pseudohermitian local frame of $M$, then any other pseudohermitian local frame $(Z',\bar Z', T')$ can be locally obtained from $(Z,\bar Z, T)$ as:
\begin{align*}
{Z'}&=\ee^{-\ii v}Z,  & {\bar Z'}&=\ee^{\ii v}\bar Z, & {T'}&=T, \\
 {\bar\zeta'}&=\ee^{-\ii v}\bar\zeta, & {\zeta'}&=\ee^{\ii v}\zeta, & {\theta'}&=\theta.
\end{align*}
for some real valued function $v$ on $M$. The functions $a$, $b$, and $c$ in the new frame are then
\begin{align*}
 a'&=\ee^{\ii v}(a- \bar Zv), &
 b'&=b+\ii Tv, &
 c'&=\ee^{2\ii v}c.
\end{align*}
\end{Rem}
\begin{Def}
A change of frame as described in  Remark~\ref{rem:change_of_frame} is called a \emph{pseudohermitian change of frame}.
\end{Def}

\begin{Rem}
On any Levi nondegenerate pseudohermitian CR manifold there is a naturally defined linear connection $\nabla^\theta$, called the Tanaka-Webster connection (see e.g.~\cite{DT}). In this paper we will not use it, however we notice that the functions $a$, $b$, and $c$ are related to $\nabla^\theta$ and to its torsion $T^\theta$ by the following equations:
\begin{gather*} \nabla^{\theta}_ZZ=-\ii\bar aZ,\quad\nabla^{\theta}_{\bar{Z}}Z=iaZ,\quad \nabla^{\theta}_TZ=-bZ,\quad T^{\theta}(T,\bar{Z})=cZ.
\end{gather*}
\end{Rem}

\begin{Rem}
A strongly pseudoconvex pseudohermitian CR manifold is in a natural way a contact metric manifold. In the three dimensional case
it is also true that every contact metric manifold is a strongly pseudoconvex CR manifold, 
since every almost CR structure is integrable. We recall that a contact metric manifold is Sasakian if and only if the contact metric structure is normal (see e.g.~\cite{DT}). This is equivalent to the vanishing of the pseudohermitian torsion $\tau=T^\theta(T,\cdot)$. With our notation, $M$ is a Sasakian manifold if and only if $c$ vanishes identically.
\end{Rem}

\subsection{CR isometric immersions}
Let $(M, T^{1,0}M, \theta)$ be a three-dimensional Levi nondegenerate pseudohermitian CR manifold and $f:M\rightarrow\mathbb{R}^m$ an isometric immersion.
Using the fact that $f$ is isometric if and only if all the conditions
\begin{equation*}
\begin{cases}
 \langle Zf,Zf\rangle=0,\\
\langle Tf,Zf\rangle=0,\\
\|Tf\|^2=\|Zf\|^2=1.
 \end{cases}
\end{equation*}
hold true for some pseudohermitian frame $(Z, \bar{Z}, T)$ (see \cite{AL}), easy computations give the following formulas which will be used several times in the sequel:
\begin{align}
\langle Zf,ZTf\rangle &= \bar c, & \langle Zf,\bar ZTf\rangle &= -\ii/2, & \langle Zf,TTf\rangle &=0, \nonumber\\
\langle Zf,Z\bar Zf\rangle &= -\ii\bar a, & \langle Zf,\bar Z\bar Zf\rangle &= -\ii a, & \langle Zf,T\bar Zf\rangle &=b-\ii/2, \label{variouseqs_CRiso}\\
\langle Tf,TZf\rangle &= 0, & \langle Tf, ZZf\rangle &= -\bar c, & \langle Tf,\bar Z Zf\rangle &=\ii/2.\nonumber
\end{align}
In the following we will identify $M$ to its isometric image $f(M)\subset\bR^n$, and denote the Euclidean standard connection on $\bR^n$  
by juxtaposition, 
the Levi-Civita connection on $M$ by $\nabla$, and the induced connection on the normal bundle $NM$ of $M$ by $D$.
The second fundamental form of the immersion is defined by:
\begin{align*}
A\colon TM\times TM&\to TM^\perp\subset\bR^n\\
  (U,V)&\mapsto A(U,V)=U Vf-(\nabla_U V)f,
\end{align*}
and the associated shape operator by:
\[
\langle A_\xi U,V\rangle=\langle\xi,A(U,V)\rangle \quad\text{for $U,V\in TM$ and $\xi\in NM$.}
\]
Let us recall the existence theorem for submanifolds. Let $M$ be an $m$-dimensional Riemannian manifold, $N$ a $k$-dimensional vector bundle over $M$ with a
connection $D$, and $A\colon TM\times TM\to N$ a symmetric bilinear map. Then there exists an isometric immersion $f\colon M \rightarrow \bR^{m+k}$
with normal bundle $N$ (up to a parallel vector bundle isometric isomorphism) and second fundamental form $A$ if and only if
the following equations of Gau\ss-Codazzi-Mainardi and Ricci are satisfied
\begin{align*}
& \langle R(U_1,U_2)U_3,U_4\rangle=\langle A(U_2,U_3),A(U_1,U_4)\rangle-\langle A(U_1,U_3),A(U_2,U_4)\rangle&&\textrm{ (Gau\ss)},\\
&(D_{U_1}A)(U_2,U_3)=(D_{U_2}A)(U_1,U_3)&&\textrm{ (Codazzi)},\\
& R^{\perp}(U_1,U_2)\xi=A(U_1,A_{\xi}U_2)-A(A_{\xi}U_1,U_2) &&\textrm{ (Ricci)}.
\end{align*}
with $R$ (resp. $R^{\perp}$) the Riemannian curvature on $M$ (resp. $N$) and $\xi\in\Gamma(N)$.

\section{Associated families}
We give now a formal definition of associated families for a three-dimensional pseudohermitian manifold.
\begin{Def}
Let $M$ be a three-dimensional pseudohermitian CR manifold, and  $(Z,\bar Z, T)$ a pseudohermitian local frame.  For all $\varphi\in\mathbb{R}$ let $r_\varphi\colon T^\bC M\to T^\bC M$ be the \emph{rotation} of angle $\varphi$ defined by:
\begin{align*}
r_{\varphi}Z=e^{\ii\varphi}Z,\quad r_{\varphi}\bar{Z}=e^{-\ii\varphi}\bar{Z},\quad r_{\varphi}T=T,
\end{align*}
A smooth one-parameter family $f_{\varphi}:M\rightarrow\mathbb{R}^{n}$ of isometric immersions
will be called an \emph{associated family} if the second fundamental forms $A^{f_{\varphi}}$ of $f_\varphi$ satisfy
\begin{equation}\label{eq_ass_fam}
\iota_{\varphi}(A^{f_\varphi}(U,V))=A^{f_0}(r_{\varphi}U,r_{\varphi}V),
\end{equation}
for all $U,V\in T^\bC M$, $\varphi\in\bR$ and for some bundle isomorphisms $\iota_{\varphi}:N^{\varphi}\rightarrow N^{0}$, where $N^{\varphi}$ is the normal bundle of $f_{\varphi}(M)$.

An isometric immersion $f\colon M\to\bR^n$ is said to \emph{admit an associated family} if there exists an associated family $f_\varphi$ such that $f_0=f$.
\end{Def}

\begin{Rem}
The definition \eqref{eq_ass_fam} is analogous to the classical definition of \emph{weak} associated families. One could try to use a different definition, more closely resembling classical \emph{strong} associated families, namely
\[
\iota_\varphi\circ df_{\varphi}(U)= df(r_{\varphi}U),\quad\text{for all $U\in TM$,}
\]
for some parallel isomorphism $\iota_\varphi\colon f^*_\varphi T\bR^n\to f^* T\bR^n$. However, from the condition $0=d(df_{\pi})=d(df\circ r_\pi)$ one easily obtains $Tf=0$, showing that there are no examples of strong associated families in this sense.
\end{Rem}
Using respectively Gau\ss, Codazzi and Ricci equations for the second fundamental forms $A^{f_\varphi}(\,\cdot\,,\,\cdot\,)$ in an associated family, we find the following
\begin{Prop}
Necessary and sufficient conditions for the existence of an associated family are given by the following equations
\begin{align}
0&=\langle A(T,Z),A(Z,\,\cdot\,)\rangle-\langle A(Z,Z),A(T,\,\cdot\,)\rangle \label{associated_gauss}\\
0&= (D A)(T,Z)=(D A)(Z,\bar{Z})=(D A)(T,\bar{Z})\label{associated_Codazzi}\\
0&=R^{\perp}(Z,T)\xi=A(A_{\xi}Z,T)-A(Z,A_{\xi}T)\label{associated_Ricci}.
\end{align}
\end{Prop}
\begin{proof}
 The proof is inspired by \cite{E}. We show in detail how to obtain equation \eqref{associated_gauss}.
We consider the Gau\ss~equation
\begin{align*}
\langle R(U_1,U_2)U_3,U_4\rangle=\langle A^{f_\varphi}(U_2,U_3),A^{f_\varphi}(U_1,U_4)\rangle-\langle A^{f_\varphi}(U_1,U_3),A^{f_\varphi}(U_2,U_4)\rangle.
\end{align*}
for all type of vectors ($Z,\bar{Z}$ or $T)$. In all possible cases the terms on the right hand side pick up a common factor
$e^{\ii k}$ for some $k\in\bR$. The right hand side vanishes by the anti-symmetry of the Riemannian curvature as soon as
$U_1=U_2$ or $U_3=U_4$ and in that case the Gau\ss~ equation is automatically satisfied. Otherwise we get (up to conjugation)
\begin{align*}
\langle R(Z,T)Z,\cdot\rangle&=\langle A(T,Z),A(Z,\cdot)\rangle-\langle A(Z,Z),A(T,\cdot)\rangle\\
&=e^{\ii\varphi}(\langle A(T,Z),A(Z,\cdot)\rangle-\langle A(Z,Z),A(T,\cdot)\rangle),
\end{align*}
which gives equation \eqref{associated_gauss}. The proof of equations \eqref{associated_Codazzi} (resp. \eqref{associated_Ricci}) is similar using the equations of Codazzi (resp. Ricci) and hence we omit them.
\end{proof}

We now want to compute these equations explicitly. For that purpose we first compute the second fundamental form.
\begin{Lem}\label{lemma_2ndFF}
The second fundamental form of an isometric immersion $f\colon M\to\bR^n$ is given by:
\begin{align*}
A(T,Z)&=TZf+(b-\frac{\ii}{2})Zf,&
A(Z,Z)&=ZZf-\ii\bar{a}Zf+\bar{c}Tf,\\
A(Z,\bar{Z})&=Z\bar{Z}f+\ii\bar{a}\bar{Z}f+\frac{\ii}{2}Tf,&
A(T,T)&=TTf.
\end{align*}
\end{Lem}
\begin{proof} As the second fundamental form $A$ is the normal part of the Euclidean connection, we have
\begin{align*}A(T,Z)=TZf-\langle TZf,Zf\rangle \bar{Z}f-\langle TZf,\bar{Z}f\rangle Zf-\langle TZf,Tf\rangle \bar{T}f.
\end{align*}
We get $\langle TZf,Zf\rangle =\langle TZf,Tf\rangle =0$ and $\langle TZf,\bar{Z}f\rangle =\frac{\ii}{2}-b$ using equations \eqref{variouseqs_CRiso}, which together give the first equation. The three other equations can be computed similarly.\end{proof}
\subsection{The Gau\ss{} equation}
\begin{Prop}
For an isometric immersion $f\colon M\to\bR^n$, the Gau\ss{} equation \eqref{associated_gauss} is equivalent to the following:
\begin{align}
T\bar{c}-\ii\bar{c}+2b\bar{c}&=0,\label{explicit_gauss1}\\
\bar{Z}\bar{c}-2\ii a\bar{c}&=0.\label{explicit_gauss2}\end{align}
\end{Prop}
\begin{proof} Inserting $T$ in equation \eqref{associated_gauss} we get, by use of Lemma \ref{lemma_2ndFF},
\begin{align*}
0=\langle A(T,Z),A(Z,T)\rangle-\langle A(Z,Z),A(T,T)\rangle=\langle TZf,TZf\rangle-\langle ZZf,TTf\rangle.
\end{align*}
Hence using equations \eqref{eq_for_a}, \eqref{eq_for_b_barc} and \eqref{eq_for_c_barb}, and finally \eqref{variouseqs_CRiso} we have
\begin{align*}
\langle TZf,TZf\rangle-\langle ZZf,TTf\rangle&=\langle TZf,TZf\rangle+\langle Zf,TZTf\rangle\\
&\qquad+\langle Zf, bZTf\rangle+\langle Zf,\bar{c}\bar{Z}Tf\rangle\\
&=\langle TZf,TZf\rangle-\langle TZf,ZTf\rangle +T\bar{c}+b\bar{c}-\frac{\ii}{2}\bar{c}\\
&=\langle TZf,TZf\rangle-\langle TZf,TZf\rangle-\langle TZf,bZf\rangle\\
&\qquad-\langle TZf,\bar{c}\bar{Z}f\rangle+T\bar{c}+b\bar{c}-\frac{\ii}{2}\bar{c}\\
&=-\ii\bar{c}+T\bar{c}+2b\bar{c},
\end{align*}
and the first equation holds. Equation \eqref{explicit_gauss2} follows in a similar way by inserting $\bar{Z}$
in \eqref{associated_gauss}.\end{proof}

We show now that in a suitable pseudohermitian frame the functions $a$, $b$ and $c$ have a very simple form.

\begin{Prop}\label{values_abc}
If an isometric immersion $f\colon M\to\bR^n$ satisfies \eqref{associated_gauss}, then there exists a pseudohermitian local frame $(Z,\bar Z,T)$ such that $a=0$, $b=\ii/2$, $c$ is constant and purely imaginary, $ac=0$, and $a$ satisfies $Ta=\frac\ii2a$.
\end{Prop}
\begin{proof}
Because of Remark \ref{rem:change_of_frame}, we can choose a pseudohermitian frame such that $c$ is purely imaginary. Thus from the first equation we obtain that either $c=0$ or $c$ is $T$-constant and $b=\ii/2$.

If $c\neq0$ is $T$-constant and $b=\ii/2$, then by equations \eqref{eq_jacobi} and \eqref{explicit_gauss1} we have:
\begin{align*}
\ii Zc+Ta-\frac{\ii a}2-\bar{a}c&=0,
\end{align*}
and thus:
\[
Ta-\frac\ii2a+\bar ac=0.
\]
Replacing $a=-(\ii\bar{Z}\bar{c})/(2\bar{c})$, one gets $2\ii Z(c\bar c)=0$, i.e. $|c|$ is constant along $Z$, and then $|c|$ and $c$ are constant. It follows then that $a=0$.

If $c=0$ then both equations \eqref{explicit_gauss1}, \eqref{explicit_gauss2} are trivially satisfied. We can, after a suitable pseudohermitian change of frame, assume that $b=\ii/2$. From \eqref{eq_jacobi} then we obtain $Ta=\frac\ii2a$.
\end{proof}

\begin{Def}
We say that a pseudohermitian frame $(Z,\bar Z,T)$ is \emph{adapted} if it is of the type described in Proposition \ref{values_abc}.
\end{Def}

\begin{Rem}
The Levi-Civita connection in an adapted pseudohermitian frame $(Z,\bar{Z},T)$ satisfies
\begin{align*}
\nabla_Z Z&=\ii\bar a Z-\bar c T,& \nabla_Z T&=\bar c \bar Z +\frac\ii2 Z,& \nabla_T Z&=0,\\
 \nabla_TT&=0,&\nabla_{\bar Z}Z&=\ii a Z+\frac\ii2T.
\end{align*}
\end{Rem}

\subsection{Codazzi equations and parallelity of the mean curvature}
We turn now our attention to equation \eqref{associated_Codazzi}, under the assumption that \eqref{associated_gauss} holds true.
\begin{Prop}\label{prop:codazzi}
If $f\colon M\to\bR^n$ is an isometric immersion for which \eqref{associated_gauss} holds true, then in an adapted pseudohermitian frame $(Z,\bar Z,T)$ the Codazzi equations \eqref{associated_Codazzi} are equivalent to the following:
\begin{align*}
D_TA(Z,Z)&=0,\\
D_{\bar{Z}}A(Z,Z)&=2\ii aZZf+2|a|^2Zf+\ii TZf ,\\
D_ZA(Z,T)&=-\bar{c}TTf+\frac{\ii}{2}(ZZf+\bar{c}Tf-\ii\bar aZf)+\bar{c}(Z\bar{Z}f+\frac{\ii}{2}Tf)+\ii\bar aTZf,\\
D_{\bar{Z}}A(Z,T)&=\frac{\ii}{2}TTf+{c}(ZZf+\bar{c}Tf)-\frac{\ii}{2}(Z\bar{Z}f+\frac{\ii}{2}Tf-\ii\bar a\bar Zf)+\ii aTZf,\\
D_TA(Z,T)&=0,\\
D_ZA(T,T)&=\ii TZf+2\bar{c}T\bar{Z}f,\\
D_ZA(Z,\bar{Z})&=-\bar{c}T\bar{Z}f-\frac{\ii}{2}TZf,\\
D_TA(Z,\bar{Z})&=0.
\end{align*}
\end{Prop}
\begin{proof}
For $X_1,X_2,X_3\in\{Z,\bar Z,T\}$ not all equal we have:
\begin{align*}
0=(D_{X_1}A)(X_2,X_3)=D_{X_1}A(X_2,X_3)-A(\nabla_{X_1}X_2,X_3)-A(X_2,\nabla_{X_1}X_3).
\end{align*}
With
\begin{align*}
\nabla_{X_i}X_j=\langle X_iX_jf,\bar{Z}f\rangle Z+\langle X_iX_jf,Zf\rangle\bar{Z}+\langle X_iX_jf,Tf\rangle T,
\end{align*}
we get
\begin{align*}
D_{X_1}A(X_2,X_3)&=\langle X_1X_2f,\bar{Z}f\rangle A(Z,X_3)+\langle X_1X_2f,Zf\rangle A(\bar{Z},X_3)\\
&\qquad+\langle X_1X_2f,Tf\rangle A(T,X_3)+\langle X_1X_3f,\bar{Z}f\rangle A(X_2,Z)\\
&\qquad+\langle X_1X_3f,Zf\rangle A(X_2,\bar{Z})+\langle X_1X_3f,Tf\rangle A(X_2,T).
\end{align*}
We give an idea of the proof by computing an example.
Let $X_2=Z$, $X_3=Z$.
Using \eqref{eq_for_a}, \eqref{eq_for_b_barc}, \eqref{eq_for_c_barb} and  Lemma \ref{lemma_2ndFF} we get, for $X_1=T$
\begin{align*}
D_TA(Z,Z)&=2\langle TZf,\bar{Z}f\rangle A(Z,Z)+2\langle TZf,Tf\rangle A(Z,T)=0,
\end{align*}
and for $X_1=\bar{Z}$,
\begin{align*}
D_{\bar{Z}}A(Z,Z)&=2\langle \bar{Z}Zf,\bar{Z}f\rangle A(Z,Z)+2\langle \bar{Z}Zf,Tf\rangle A(Z,T)\\
&=2\ii aZZf+2|a|^2Zf+\ii TZf,
\end{align*}
which gives the first two equations by use of Lemma \ref{values_abc}. The six other cases can be computed in a similar way.
\end{proof}
\begin{Prop} \label{rem_meancurv}If an isometric  immersion $f\colon M\to\bR^n$ admits an associated family, then the mean curvature vector of $f(M)$ is parallel in the normal bundle of $f(M)$.
\end{Prop}
\begin{proof}
The mean curvature vector is $H=2A(Z,\bar Z)+A(T,T)$. From
\[
D_ZA(T,T)=-2D_ZA(Z,\bar{Z}),
\]
we have that $H$ is parallel.
\end{proof}

\subsection{Ricci equations and parallelity of the second fundamental form}
In this section we prove that if $f\colon M\to\bR^n$ belongs to an associated family, then the second fundamental form is parallel. It follows then that $f(M)$ is locally extrinsic symmetric in $\bR^n$.
\begin{Prop}
If an isometric immersion $f\colon M\to \bR^n$ admits an associated family,
then the second fundamental form of the immersion is parallel, that is $DA=0$.
\end{Prop}
\begin{proof}
We fix an adapted pseudohermitian frame $(Z,\bar Z, T)$.
Note that by the symmetry of the Codazzi equations \eqref{associated_Codazzi}, $(D_UA)(V,W)=0$ whenever $U,V,W\subset\{Z,\bar Z, T\}$ are not all equal.
From equation \eqref{associated_Ricci} we have
\[
D_{[\bar Z, T]}(A(Z,Z))=[D_{\bar Z},D_T](A(Z,Z)).
\]
Moreover:
\begin{align*}
D_{[\bar Z, T]}(A(Z,Z))&=-bD_{\bar Z}(A(Z,Z))+cD_{Z}(A(Z,Z))\\
  &=\big(\frac12-2|c|^2\big)A(Z,T)+c(D_ZA)(Z,Z)+a A(Z,Z),
  \end{align*}
  and
  \begin{align*}
[D_{\bar Z},D_T](A(Z,Z))&=D_{\bar Z}(D_T(A(Z,Z)))-D_T(D_{\bar Z}(A(Z,Z)))\\
  &=2D_{\bar Z}(A(\nabla_TZ,Z))-2D_T(A(\nabla_{\bar Z}Z,Z))\\
  &=-2A(\nabla_T\nabla_{\bar Z}Z,Z)=-2\ii Ta A(Z,Z)-\ii A(\nabla_TT,Z)\\
  &=a A(Z,Z),
\end{align*}
so that:
\begin{align}\label{computation_(D_ZA(ZZ))}
c(D_ZA)(Z,Z)&=2(|c|^2-\frac14)A(Z,T).
\end{align}

We consider now the term $(D_TA)(T,T)$. We have:
\begin{align*}
(D_TA)(T,T)&=D_TA(T,T)=\ii D_{[Z,\bar Z]}A(T,T)\\
  &=\ii[D_Z,D_{\bar Z}]A(T,T)-\ii R^\perp(Z,\bar Z)A(T,T)\\
  &=2\ii A(\nabla_Z\nabla_{\bar Z}T,T)-2\ii A(\nabla_{\bar Z}\nabla_{Z}T,T)-\ii R^\perp(Z,\bar Z)A(T,T)\\
  &=-\ii R^\perp(Z,\bar Z)A(T,T)\\
  &=-\ii A(A_{TTf}\bar{Z},Z)+\ii A(A_{TTf}Z,\bar{Z})\\
&=-\ii\langle \bar{Z}Zf,TTf\rangle A(\bar{Z},Z)-\ii\langle \bar{Z}\bar{Z}f,TTf\rangle A(Z,Z)\\
&\qquad-\ii\langle \bar{Z}Tf,TTf\rangle A(T,Z)+\ii\langle Z\bar{Z}f,TTf\rangle A(Z,\bar{Z})\\
&\qquad+\ii\langle ZZf,TTf\rangle A(\bar{Z},\bar{Z})+\ii\langle ZTf,TTf\rangle A(T,\bar{Z}).
\end{align*}

We now consider separately the three cases $c=0$, $c=\frac\ii2$ and $c\neq0,\frac\ii2$.

\paragraph{Case $c=0$.} 
We have $A(Z,T)=TZ=0$. Replacing in $D_ZA(Z,T)$, we get $ZZ=\ii\bar aZ$ which implies $A(Z,Z)=0$ and $(D_ZA)(Z,Z)=0$.
Finally,
\[
(D_TA)(T,T)=\ii\langle [\bar{Z},Z]f,TTf\rangle A(\bar{Z},Z)=0,
\]
and thus $DA=0$.

\paragraph{Case $c=\frac\ii2$.}
From \eqref{computation_(D_ZA(ZZ))} it follows that $(D_ZA)(Z,Z)=0$.
From $D_ZA(Z,T)=D_{\bar Z}A(Z,T)$ we obtain
\begin{align*}
0&=\frac\ii2D_{Z-\bar Z}A(Z,T)=D_{[Z,T]}A(Z,T)=[D_Z,D_T]A(Z,T)\\
  &=-D_TD_ZA(Z,T)=-D_T\big(\frac\ii2 A(T,T)+\frac\ii2A(Z,Z)-\frac\ii2A(Z,\bar Z)\big)\\
  &=-\frac\ii2D_TA(T,T)=-\frac\ii2(D_TA)(T,T),
\end{align*}
and hence $DA=0$.
\paragraph{Case $c\neq0,\frac\ii2$.}
If $c\neq 0$, from $D_TA(Z,T)=0$ and equation \eqref{computation_(D_ZA(ZZ))}, we have:
\begin{align*}
0&=D_T\big((D_ZA)(Z,Z)\big)\\
&=D_T\big(D_ZA(Z,Z)-2A(\nabla_ZZ,Z)\big)=D_TD_ZA(Z,Z)+2\bar{c}D_TA(Z,T)\\
&=D_ZD_TA(Z,Z)+D_{[T,Z]}A(Z,Z)=-\frac\ii2D_ZA(Z,Z)-\bar{c}D_{\bar{Z}}A(Z,Z)\\
&=-\frac\ii2(D_ZA)(Z,Z)-2bA(\nabla_ZZ,Z)-\bar{c}(D_{\bar{Z}}A)(Z,Z)-2\bar{c}A(\nabla_{\bar{Z}}Z,Z)\\
&=-\frac{\ii}{2}(D_ZA)(Z,Z)+\ii\bar{c}A(Z,T)-\ii\bar{c}A(Z,T).
\end{align*}
Hence $(D_ZA)(Z,Z)=0$ and, since
$c\neq\frac{\ii}{2}$, we have again $TZf=0=A(Z,T)$, and finally:
\[
(ZZf,TTf)=-(TZZf,Tf)=b(ZZf,Tf)+\bar{c}(\bar{Z}Zf,Tf)=0,
\]
i.e. $(D_TA)(T,T)=0$ and thus $DA=0$.
\end{proof}

\subsection{The classification}
 In this section we prove the main statement of this paper:
\begin{Th}\label{main_theorem}
Let $M$ be a three-dimensional strongly pseudo-convex pseudohermitian CR-manifold, and $f: M\rightarrow\mathbb{R}^n$ an isometric immersion admitting an associated family. Then one of the two following cases occurs:
\begin{itemize}
\item There exists an  isometric  local CR-diffeomorphism $\tilde f$ from $M$ to the standard sphere $S^3\subset\bC^2$ of radius $\sqrt2$ and an affine isometric embedding $\psi\colon\bC^2\to\bR^n$ such that $f=\psi\circ\tilde f$.
\item There exists an  isometric  local CR-diffeomorphism $\tilde f$ from $M$ to the cylinder $S^1+\ii\bR^2\subset\bR^2\oplus\ii\bR^2\simeq\bC^2$ of radius $1$ and $f(M)$ is an open domain in a three-dimensional affine subspace of $\bR^n$.
\end{itemize}
\end{Th}

\begin{Cor}\label{main_corollary}
Let $M$ be a complete three-dimensional strongly pseudo-convex pseudohermitian CR-manifold, and $f: M\rightarrow\mathbb{R}^n$ an isometric immersion admitting an associated family. Then one of the two following cases occurs:
\begin{itemize}
\item There exists an  isometric  CR-diffeomorphism $\tilde f$ from $M$ onto the standard sphere $S^3\subset\bC^2$ of radius $\sqrt2$ and an affine isometric embedding $\psi\colon\bC^2\to\bR^n$ such that $f=\psi\circ\tilde f$.
\item There exists an  isometric covering map and  local CR-diffeomorphism $\tilde f$ from $M$ to the cylinder $S^1+\ii\bR^2\subset\bR^2\oplus\ii\bR^2\simeq\bC^2$ of radius $1$ and $f(M)$ is a three-dimensional affine subspace of $\bR^n$.\qed
\end{itemize}
\end{Cor}
We fix an adapted pseudohermitian frame $(Z,\bar Z, T)$.
The first step of the proof is computing the sectional curvatures $K(U,V)=\langle R(U,V)\bar V,\bar U\rangle$ of $M$:
\begin{align}\label{eq_sect_curv}
\begin{gathered}
K(Z,\bar Z)=K(X,Y)=|c|^2-\frac14-2|a|^2-\ii Za-\ii\bar Z\bar a,\\
K(Z,T)=K(\bar Z,T)=K(X,T)=K(Y,T)=\frac14-|c|^2.
\end{gathered}
\end{align}

The submanifolds of Euclidean spaces with parallel second fundamental form are exactly the (locally) extrinsic symmetric submanifolds.
By \cite[Theorems 2 and 5]{BackesReckziegel1983}, they are locally isometric to a product of an Euclidean space and spheres. In particular they have nonnegative sectional curvature.
By imposing that the sectional curvature is nonnegative in \eqref{eq_sect_curv}, we first obtain $|c|\leq\frac12$.
Moreover, considering that $a=0$ if $|c|\neq0$, we conclude that either $c=0$ or $|c|=\frac12$.

The next step is showing that the normal bundle of $f(M)$ in $\bR^n$ is flat.

\begin{Def}
We recall that:
\begin{itemize}
\item
The \emph{first normal bundle} $N^1M$ of $M$ is the linear span, in the normal bundle of $M$, of the image of the second fundamental form. It is a vector subbundle of $NM$ on open subsets of $M$ where it has constant rank.
\item
The \emph{osculating bundle} of $M$ is the subbundle $OM=TM+N^1M$ of $f^*T\bR^n$. It is a vector subbundle of $f^*T\bR^n$ precisely where $N^1M$ is a vector subbundle of $NM$.
\end{itemize}
\end{Def}

We have:
\begin{Lem}
The osculating bundle of $M$ has constant rank and is a parallel vector subbundle of $f^*T\bR^n$.
\end{Lem}
\begin{proof}
First we observe that
the first normal bundle $N^1M$ of $M$ has constant rank on $M$ and is a parallel vector subbundle of $NM$ because the second fundamental form of $M$ is parallel.

The osculating bundle is then a vector bundle on $M$ and the conclusion follows from 
\cite{Erbacher1971b}.
%
\end{proof}

It follows then that $M$ is contained in an affine subspace of $\bR^n$ parallel to $TM+N^1M$. By replacing $\bR^n$ with this subspace we can henceforth assume without loss of generality that $NM=N^1M$.

\begin{Prop}
The normal bundle of $M$ is flat.
\end{Prop}
\begin{proof}
By equation \eqref{associated_Ricci} we have $R^\perp(Z,T)=0$ and consequently $R^\perp(\bar Z, T)=0$.

Explicit computations using 
Proposition \ref{prop:codazzi}, show that $R^\perp(Z,\bar Z)A(U,V)=0$ for all $U,V\in TM$, if $c=0,\frac\ii2$.
Since the normal bundle is spanned by vectors of the form $A(U,V)$, the normal bundle is flat.
\end{proof}

We also recall that the mean curvature of $f(M)$ in $\bR^n$ is $H=2A(Z,\bar Z)+A(T,T)$. By Remark \ref{rem_meancurv}  the mean curvature vector is parallel in the normal bundle.
By \cite[Theorem~1]{Erbacher1971a}  the immersion is locally a product of standard immersions $\bR\to\bR$, $S^1\to\bR^2$, $S^2\to\bR^3$, or $S^3\to\bR^4$.

To conclude the proof of Theorem \ref{main_theorem}, we consider separately the two cases $c=0$ and $c=\frac{\ii}{2}$.

\paragraph{Case $c=0$.}
The manifold $M$ is not flat, hence  it is locally isometric either to a $3$-dimensional sphere or to a product $S^2\times\bR$.
We argue  that $M$ cannot be locally isometric to $S^2\times\bR$. If it was, then there would be a nonzero vector $U\in TM$ such that $K(U,V)=0$ for every tangent vector $V\in TM$.
However, $K(X,T)\neq0$, so U cannot be proportional to $T$, and $K(U,T)=\frac14$ for every unit vector $U$ in the plane generated by $X$ and $Y$.

The manifold $M$ is then locally isometric to the sphere. By \cite{BV1987} the only Sasakian structure on the round three-sphere is the standard one, that is  the CR structure and the metric are those induced by the embedding in $\bC^2$ as
\[ S^3=\{z\in\bC^2\mid\|z\|^2=2\},\]
and the pseudohermitian form $\theta$ is given by $\theta(v)=\frac1{\sqrt{2}}\Im\langle \bar z,v\rangle$ for all $z\in S^3$ and $v\in T_zS^3$.
The theorem is proved in this case.

\paragraph{Case $c=\frac{\ii}{2}$.}
In this case $a=0$, all the sectional curvatures vanish, and the manifold $M$ is flat. Again by \cite[Theorem~1]{Erbacher1971a} the immersion is a product of standard immersions of factors of the form $\bR$ or $S^1$. We can then assume that the codimension is $3$, i.e. that $n=6$.
Up to isometries of $\bR^6$, the image $f(M)$ is (locally) the submanifold:
\[ T_{\lambda}=\left\{x\in\bR^6\left\vert \begin{aligned}&x_{2j-1}^2+(x_{2j}-\lambda_j^{-1})^2=\lambda_j^{-2},&&\text{if $\lambda_j\neq0$,}\\
&x_{2j-1}=0,&&\text{if $\lambda_j=0$,}\end{aligned}\quad j=1\ldots3\right\}\right. \]
for some $\lambda=(\lambda_1,\lambda_2,\lambda_3)$ with $0\leq\lambda_1\leq\lambda_2\leq\lambda_3$.

We recall that for an associated family $f_\varphi$, the second fundamental form satisfies the identity:
\begin{equation}\label{eq_ass_fam2}
\iota_{\varphi}(A^{f_\varphi}(U,V))=A(r_{\varphi}U,r_{\varphi}V),
\end{equation}
for all $U,V\in TM$. The image $f_\varphi(M)$ is then again of the form $T_{\lambda(\varphi)}$. The vector valued bilinear forms $A$ and $A^{f_\varphi}$ are conjugated under the action of $\mathrm{O}(TM)\times\mathrm{O}(NM)$ if and only if
$\lambda(\varphi)=\lambda$. We then conclude that $\lambda(\varphi)=\lambda$ is constant for all $\varphi$.

\begin{Lem}
For all $0\leq\lambda_1\leq\lambda_2\leq\lambda_3$, with $\lambda_2>0$,
the Lie algebra $\mathfrak{isom}(\bR^6,T_\lambda,0)$ of germs of Killing vector fields on $\bR^6$, tangent to $T_\lambda$ and vanishing in $0$ is $\{0\}$.
\end{Lem}
\begin{proof}
First we notice that every germ of Killing vector field in $\mathfrak{isom}(\bR^6,T_\lambda,0)$ extends to a global Killing vector field on $\bR^6$, tangent to $T_\lambda$ and vanishing in $0$.

If $\lambda_1>0$, the set of points in $\bR^6$ equidistant from all points in $T_\lambda$ consists of the single point $o=(0,\lambda_1^{-1},0,\lambda_2^{-1},0,\lambda_3^{-1})$. Every element of $\mathfrak{isom}(\bR^6,T_\lambda,0)$ hence vanishes in $o$.

\begin{sloppypar}
Let $\tilde v\in\mathfrak{isom}(\bR^6,T_\lambda,0)$ and $v=\tau_o\circ\tilde v\circ\tau_{-o}$, where $\tau_o$ and $\tau_{-o}$ denote the translations of vector $o$ and $-o$ respectively. Then $v\colon\bR^6\to\bR^6$ is a linear map in $\mathfrak{so}(\bR^6)$, conformal with respect to the symmetric bilinear forms $\mathrm{diag}(-2\lambda_1,-2\lambda_1,\lambda_2,\lambda_2,\lambda_3,\lambda_3)$, $\mathrm{diag}(\lambda_1,\lambda_1,-2\lambda_2,-2\lambda_2,\lambda_3,\lambda_3)$ and $\mathrm{diag}(\lambda_1,\lambda_1,\lambda_2,\lambda_2,-2\lambda_3,-2\lambda_3)$, and vanishing on $o$. It easily follows that $v=0$.
\end{sloppypar}

The case $\lambda_1=0$ is treated similarly.
\end{proof}


\begin{Lem}
For an associated family $f_\varphi$, we have $\lambda_1=\lambda_2=0$.
\end{Lem}
\begin{proof}
Since all the images $f_\varphi(M)$ are locally congruent up to isometries of $\bR^6$, we can rewrite \eqref{eq_ass_fam2} as
\[
d\psi_{\varphi}(A(d\psi_\varphi(U),d\psi_\varphi(V)))=A(r_{\varphi}U,r_{\varphi}V),
\]
for some $\psi_\varphi\in\mathrm{Isom}(\bR^6)$ that preserves $T_\lambda$.

Assume $\lambda_2>0$.
The previous Lemma implies that $\psi_{\varphi}$ is (conjugated to) the identity.
It follows that $\frac{A(U,U)}{\|U,U\|}$ is constant for $U\in HM\setminus\{0\}$, hence $A|_{HM}=\lambda I \otimes\xi$ for some $\lambda\geq0$ and $\xi\in NM$ with $\|\xi\|=1$. Actually $\lambda>0$ because there exists some $\xi\in NM$ such that $A_{\xi}$ has signature $(2,0)$.

There exist three orthonormal vectors $\xi_1,\xi_2,\xi_3$ such that $\mathrm{rk}\,A_{\xi_j}=1$, for $j=1,2,3$, and thus $A_{\xi_j}|_{HM}=0$. But this would imply that $A_\xi|_{HM}=0$, yielding a contradiction. We can then conclude that $\lambda_2=\lambda_1=0$.
\end{proof}

\begin{Lem}
If the immersion $f$ admits an associated family $f_\varphi$, then $\lambda_1=\lambda_2=\lambda_3=0$.
\end{Lem}
\begin{proof}
Assume $0=\lambda_1=\lambda_2<\lambda_3$. Then $M=\bR^2\times S^1_{\lambda_3}\subset\bR^4$. We fix an unit normal vector field $\xi$. Then $\mathrm{rk}\,A_\xi=1$. Since the codimension of $M$ is $1$, and $\frac{A_\xi(U,U)}{\|U,U\|}$ is constant for $U\in HM\setminus\{0\}$, it follows also in this case that $\frac{A(U,U)}{\|U,U\|}$ is a constant multiple of $\xi$ for $U\in HM\setminus\{0\}$, and that $A_{HM}=0$. However the subbundle of totally isotropic vectors for $A_\xi$ is integrable, and $HM$ is not, giving a contradiction.\end{proof}

Finally, if $\lambda_1=\lambda_2=\lambda_3=0$ then $f(M)$ is locally a three-dimensional affine subspace, and \eqref{eq_ass_fam2} is trivially satisfied. The theorem is proved also in the case $c=\frac\ii2$.


\begin{thebibliography}{BEFT}
\bibitem[AL]{AL} A. Altomani, M.-A. Lawn,
  \textit{Isometric and CR pluriharmonic immersions of three dimensional CR manifolds in Euclidean spaces},
   to appear in Hokkaido Math. J., preprint (2011)
\bibitem[BEFT]{BEFT}F.E. Burstall, J.-H. Eschenburg, M.J. Ferreira, R. Tribuzy,
  \textit{K\"ahler submanifolds with parallel pluri-mean curvature},
  Differential Geometry and its Applications, Volume 20, Issue 1, 47-66 (2004).
\bibitem[BR]{BackesReckziegel1983}E. Backes, H. Reckziegel,
  \textit{On symmetric submanifolds of spaces of constant curvature},
  Math. Ann. \textbf{263} no. 4, 419-433 (1983).
\bibitem[BV]{BV1987}D. Blair, L. Vanhecke,
  \textit{Symmetries and $\varphi$-symmetric spaces},
  T\^ohoku Math. J., II. Ser. \textbf{39}, 373-383 (1987).
\bibitem[DT]{DT} S. Dragomir, G. Tomassini,
  \textit{Differential geometry and analysis on CR manifolds},
   Progress in Mathematics 246. Basel: Birkh\"auser (2006).
\bibitem[Er1]{Erbacher1971b} J. Erbacher,
  \textit{Reduction of the codimension of an isometric immersion},
  J. Differential Geometry \textbf5  333-340 (1971).
\bibitem[Er2]{Erbacher1971a} J. Erbacher,
  \textit{Isometric immersions of constant mean curvature and triviality of the normal connection},
  Nagoya Math. J. \textbf{45}, 139-165  (1972).
\bibitem[Es]{E}   J.-H. Eschenburg,
  \textit{The associated family},
  Mat. Contemp. \textbf{31}, 01-12 (2006).
\end{thebibliography}
\end{document}